\documentclass[a4paper, 12pt, final]{amsart}

\makeatletter
\numberwithin{equation}{section}
\numberwithin{figure}{section}
\usepackage[a4paper]{geometry}
\usepackage{amsmath,amsfonts,amssymb,amsthm,epsfig}

\usepackage{color}
\usepackage[final,linkcolor = blue,citecolor = blue,colorlinks=true]{hyperref}
\usepackage[leqno]{amsmath}

\usepackage{amscd}
\usepackage{mathrsfs}
\usepackage{}
\usepackage[all]{xy}
\usepackage[OT1]{fontenc}
\usepackage[latin1]{inputenc}
\usepackage{amssymb,latexsym}
\usepackage{type1cm,ae}
\usepackage[notref,notcite]{showkeys}
\usepackage{graphicx}
\usepackage{xspace}
\usepackage{setspace}
\usepackage{enumerate}
\usepackage[english]{babel}

\theoremstyle{definition}
		\newtheorem{theorem}{Theorem}[section]
				\newtheorem{proposition}[theorem]{Proposition}
				
				\newtheorem{lemma}[theorem]{Lemma}

	            \newtheorem{remark}[theorem]{Remark}

\numberwithin{equation}{section}

\newcommand*{\Wert}{\mathord{\mbox{|\kern-1.5pt|\kern-1.5pt|}}}

\newcommand{\N}{\mathbb{N}}	
\newcommand{\R}{\mathbb{R}}	

\newcommand{\osc}{\mathrm{osc}}

\DeclareMathOperator{\dist}{dist}

\def\XXint#1#2#3{{\setbox0=\hbox{$#1{#2#3}{\int}$}
  \vcenter{\hbox{$#2#3$}}\kern-.5\wd0}}

\makeatother

\usepackage{babel}
\begin{document}
	
\title[Boundary Regularity]{Boundary Regularity for an even order elliptic system in the critical dimension} 

\author[M.-L. Liu and Y.-L. Tian]{Ming-Lun Liu and Yao-Lan Tian*}         

\address[Ming-Lun Liu]{Research Center for Mathematics and Interdisciplinary Sciences, Shandong University, Qingdao 266237, P. R. China and Frontiers Science Center for Nonlinear Expectations, Ministry of Education, Qingdao, P. R. China} 
\email{minglunliu2021@163.com}

\address[Yao-Lan Tian]{Center for Optics Research and Engineering, Shandong University, Qingdao 266237, P. R. China}
\email{tianylbnu@126.com}

\thanks{Corresponding author: Yao-Lan Tian.}
\thanks{Both authors are partially supported by the Young Scientist Program of the Ministry of Science and Technology of China (No.~2021YFA1002200), the National Natural Science Foundation of China (No.~12101362) and the Natural Science Foundation of Shandong Province (No.~ZR2022YQ01, ZR2021QA003).}

\begin{abstract}
	In this short note,  we consider the Dirichlet problem associated to an even order elliptic system with antisymmetric first order potential. Given any continuous boundary data, we show that weak solutions are continuous up to boundary.
\end{abstract}

\maketitle

{\small    
	\smallskip
	\newline
}
\bigskip

\section{Introduction}
In this paper, we consider the Dirichlet problem for the following even order elliptic system for $u\in W^{k,2}(\Omega,\R^m)$:
\begin{equation}\label{eq:D-G system}
\Delta^k u = \sum_{l=0}^{k-1}\Delta^l\left\langle V_l,du\right\rangle+ \sum_{l=0}^{k-2}\Delta^l\delta(w_ldu)\quad \text{in } \Omega\subset \R^{2k}
\end{equation}
with the following regularity assumptions on the coefficients: 
\begin{equation}\label{eq:Assumption 1}
	\begin{aligned} & w_{i}\in W^{2i+2-k,2}\left(\Omega,\mathbb{R}^{m\times m}\right)\text{ for }i\in\{0,\ldots,k-2\},\\
		& V_{i}\in W^{2i+1-k,2}\left(\Omega,\mathbb{R}^{m\times m}\otimes\wedge^{1}\mathbb{R}^{2k}\right)\text{ for }i\in\{1,\ldots,k-1\},
	\end{aligned}
\end{equation}
and $$V_{0}=d\eta+F$$ with
\begin{equation}\label{eq:Assumption 2}
	\eta\in W^{2-k,2}\left(\Omega,so(m)\right)\quad \text{  and  }\quad  F\in W^{2-k,\frac{2k}{k+1},1}\left(\Omega,\mathbb{R}^{m\times m}\otimes\wedge^{1}\mathbb{R}^{2k}\right).
\end{equation}

This system was initially introduced by de Longueville and Gastel \cite{deLongueville-Gastel-2019}, aiming at a further extesion of the second order theory by Rivi\`ere \cite{Riviere-2007} (corresponding to the case $k=1$) and the fourth order theory by Lamm-Rivi\`ere \cite{Lamm-Riviere-2008} (corresponding to the case $k=2$), addressing an open problem of Rivi\`ere. It includes the Euler-Lagrange equations of many interesting classes of geometric mappings such as the harmonic mappings, biharmonic mappings, polyharmonic mappings and so on; see \cite{Chang-W-Y-1999,Wang-2004-CPAM,Riviere-2007,Lamm-Riviere-2008,Gastel-Scheven-2009CAG,Guo-Xiang-2019-Higher,GXZ2022}. 

A distinguished feature of this system is the \emph{criticality}. To see it, we consider the simpler case $k=1$. Then system \eqref{eq:D-G system} reduces to the second order Rivi\`ere system
\begin{equation}\label{eq:Riviere system}
	\Delta u=\Omega'\cdot \nabla u,
\end{equation}
where $u\in W^{1,2}(\Omega,\R^m)$ and $\Omega'\in L^2(\Omega,so(m)\otimes \Lambda^1\R^2)$.  The right hand side of \eqref{eq:Riviere system} is merely in $L^1$ by H\"older's inequality and so standard $L^p$ regularity theory for elliptic equations fails to apply here. In the celebrated work \cite{Riviere-2007}, Rivi\`ere succeeded in rewriting \eqref{eq:Riviere system} into an equivalent conservation law, from which the continuity of weak solutions follows. The techniques were further extended to fourth order system in \cite{Lamm-Riviere-2008} and finally to general even order systems in \cite{deLongueville-Gastel-2019}.

In this paper, we shall consider the Dirichlet boundary value problem for \eqref{eq:D-G system}. Recall that we say that $u\in W^{k,2}(\Omega,\R^m)$ has Dirichlet boundary value $g\in C^{k-1}(\overline{\Omega},\R^m)$ if  
\[
\nabla^{\alpha}u=\nabla^\alpha g\quad \text{on }\partial\Omega 
\]
holds in the sense of traces for all $2k$-dimensional multi-indices $\alpha$ with $|\alpha|\leq k-1$. Similarly, we say that $u$ has Navier boundary value $h_{i}\in C(\overline{\Omega},\R^m)$, $i=0,\cdots,k-1$, if for all $i\in \{0,\cdots,k-1\}$
\[
\Delta^i u=h_i \quad \text{on }\partial\Omega. 
\]

Now, we can state our main theorem.
\begin{theorem}\label{thm:main}
Fix $k\in \N$ and $\Omega\subset \R^{2k}$ a bounded smooth domain. Suppose $u\in W^{k,2}(\Omega,\R^m)$ is a solution of \eqref{eq:D-G system} with \eqref{eq:Assumption 1} and \eqref{eq:Assumption 2}. If either the Dirichlet boundary value $g\in C^{k-1}(\overline{\Omega},\R^m)$ or the Navier boundary value $h_i$ for $i=0,\cdots, k-1$, then $u\in C(\overline{\Omega},\R^m)$. 
\end{theorem}

Theorem \ref{thm:main} can be viewed as a natural extension of the corresponding boundary continuity results of M\"uller-Schikorra \cite{Muller-Schikorra-2009} for second order system, and Guo-Xiang \cite{Guo-Xiang-2019-Boundary} for fourth order system.
As a special case of Theorem \ref{thm:main}, we infer that every (extrinsic or intrinsic) polyharmonic mapping from the unit ball $B^{2k}\subset \R^{2k}$ into a closed manifold $N\hookrightarrow \R^m$ is continuous up the boundary, under the Dirichlet boundary value condition. This partially extends the corrosponding boundary continuity reuslt of Lamm-Wang \cite{Lamm-Wang-2009}.

The approach to Theorem \ref{thm:main} is similar to that of Lamm and Wang \cite{Lamm-Wang-2009}, relying on interior H\"older regularity and a boundary maximal principle. As noticed in \cite{Guo-Xiang-2019-Boundary}, this approach only requies zero order boundary assumption ``$u=g$" or ``$u=h_0$" on $\partial\Omega$ for some continuous function $g$ or $h_0$. This observation extends to the general even order system \eqref{eq:D-G system}. 

Our natations are rather standard. We write $B_r(x)$ for a ball centred at $x$ with radius $r$ in $\R^{2k}$. The notation $C$ denotes various constants that may be different from line to line. We sometimes write $A\lesssim B$ meaning that $A\leq CB$ for some constant $C>0$ depending only on the quantitative data.   

\section{The proof of main result}

\subsection{Interior regularity}
Continuity of weak solutions for \eqref{eq:D-G system} was first obtained in \cite{deLongueville-Gastel-2019}. But for boundary continuity, we need the stronger interior H\"older regularity. We first recall the following interior H\"older regularity result for \eqref{eq:D-G system} from \cite[Theorem 1.3]{Guo-Xiang-2019-Higher} or \cite[Theorem 1.1]{GXZ2022}. 
\begin{theorem}[Interior Regularity]\label{thm:interior regularity}
Suppose $u\in W^{k,2}(B^{2k},\R^m)$ is a solution of \eqref{eq:D-G system} with \eqref{eq:Assumption 1} and \eqref{eq:Assumption 2}. Then there exist $\alpha\in (0,1)$, $C>0$ and $r_0$, depending only on $k$, $m$ and the data from \eqref{eq:Assumption 1} and \eqref{eq:Assumption 2}, such that $u$ is locally $\alpha$-H\"older continuous and 
\[
\osc_{B_r(x)}u\leq Cr^{\alpha}\|u\|_{W^{k,2}(B^{2k})}
\]
for all $x\in B_{\frac{1}{4}}(0)$ and all $0<r<r_0$. 
\end{theorem}

We shall use the following version of Theorem \ref{thm:interior regularity} in our later proofs. There exists $R_0>0$ sufficiently small, such that, if $x\in \Omega$ and $0<r<\min\{R_0,\dist(x,\partial \Omega)/4 \}$, then $u\in C^{0,\alpha}(B_r(x_0),\R^m)$ for some $\alpha\in (0,1)$ and
\begin{equation}\label{eq:interior regularity}
	\osc_{B_{\tau r}(x)}u\lesssim \tau^{\alpha}\|u\|_{W^{k,2}(B_{4r}(x_0))}\quad \text{for }0<\tau\leq 1. 
\end{equation}
By \cite[Theorem 1.1]{GXZ2022}, one can indeed infer that \eqref{eq:interior regularity} holds for all $\alpha\in (0,1)$. But for our purpose, the current estimate is sufficient.

\subsection{Boundary Maximum Principle}

Another ingredient for boundary regularity is a boundary maximum principle, originally discovered by Qing \cite{Qing1993} in his proof of boundary regularity for weakly harmonic maps and was later adapted to the polyharmonic case in Lamm-Wang \cite{Lamm-Wang-2009}.

For $x\in \Omega$ and $R>0$, denote by $\Omega_R(x)=\Omega\cap B_R(x)$. We shall prove the following boundary maximal principle for solutions of \eqref{eq:D-G system}.
\begin{proposition}[Boundary Maximum Principle]\label{prop:BMP}
	There exists a constant $C>0$ such that for any $x\in \overline{\Omega}$ and $0<R<R_0/4$, for any $q\in\R^m$, there holds
	\begin{equation}\label{eq:boundary mp}
		\max\limits_{\Omega_R(x)}|u-q|\leq C\left( \max\limits_{\partial \Omega_R(x)}|u-q| + \|u\|_{W^{k,2}(\Omega_{4R}(x))}\right) .
	\end{equation}
\end{proposition}

To prove Proposition \ref{prop:BMP}, we need the following version of Courant-Lebesgue lemma, which was essentially established in \cite{Lamm-Wang-2009}. Since the formulation is slightly different from there, for the convenience of the readers, we recall the proofs here. 
\begin{lemma}\label{lemma:Courant-Lebesgue}
	There exists $C>0$ such that for any $R>0$ and any $x_0\in\overline{\Omega}$, there exists $R_1\in(R,2R)$ so that
	$$
	\osc_{\partial B_{R_1}(x_0)\cap\Omega}u\leq C\|u\|_{W^{k,2}(\Omega_{4R}(x_0))}.
	$$
\end{lemma}
\begin{proof}
	For $x\in B_{2R}(x_0)$, set $r = |x-x_0|\in[0,2R]$. By Fubini's theorem, we have
	\begin{align*}
		\int_{\Omega_{2R}(x_0)}|\nabla u|^{2k}dx \geq & 
		\int_{R}^{2R}dr \int_{\partial B_r(x_0)\cap\Omega}|\nabla_T u|^{2k} d\mathcal{H}^{2k-1} \\
		\geq & \left( \int_{R}^{2R}\frac1rdr\right) \inf\limits_{R\leq r\leq 2R}\left( r\int_{\partial B_r(x_0)\cap\Omega}|\nabla_T u|^{2k} d\mathcal{H}^{2k-1}\right) \\
		=&  \ln2\inf\limits_{R\leq r\leq 2R}\left( r\int_{\partial B_r(x_0)\cap\Omega}|\nabla_T u|^{2k} d\mathcal{H}^{2k-1}\right),
	\end{align*}
	where $\nabla_T$ denotes the gradient operator on $\partial B_r(x_0)$ and $d\mathcal{H}^{2k-1}$ is the volume element on $\partial B_r(x_0)$. Then there exists $R_1\in(R,2R)$ such that 
	$$
	R_1\int_{\partial B_{R_1}(x_0)\cap\Omega}|\nabla_T u|^{2k} dH^{2k-1}\leq \frac{1}{\ln 2} \int_{\Omega_{2R}(x_0)}|\nabla u|^{2k}dx.
	$$
	Hence $u(R_1,\cdot)\in W^{1,2k}(\partial B_{R_1}(x_0)\cap\Omega,\R^m)$ and the Sobolev embedding theorem implies that $u(R_1,\cdot)\in C^{\frac{1}{2k}}(\partial B_{R_1}(x_0)\cap\Omega,\R^m)$ and 
	\begin{align*}
		\osc_{\partial B_{R_1}(x_0)\cap\Omega}u\lesssim  R_1\int_{\partial B_{R_1}(x_0)\cap\Omega}|\nabla_T u|^{2k} dH^{2k-1}\lesssim \|u\|_{W^{k,2}(\Omega_{4R}(x_0))}.
	\end{align*}
\end{proof}

\begin{proof}[Proof of Proposition \ref{prop:BMP}]
	Denote $M = \max\limits_{\Omega_R(x)}|u-q|$, here $q\in\R^m$ is fixed. We may assume that 
	\begin{equation}\label{BMP01}
		M\geq \|u\|_{W^{k,2}(\Omega_{R}(x))}.
	\end{equation}
	Choose $x_0\in \Omega_R(x)$ such that 
	\begin{equation}\label{BMP02}
		|u(x_0)-q|\geq\frac34M.
	\end{equation}
	Let $r_0 = \dist(x_0,\partial\Omega_R(x_0))$. Note that $r_0\leq R\leq R_0$. Thus \eqref{eq:interior regularity} implies that for any $r\in(0,\frac{r_0}{4})$, we have
	\begin{equation}\label{BMP03}
		\osc_{B_{r}(x_0)}u \leq C\left( \frac{r}{r_0}\right) ^{\alpha_0} \|u\|_{W^{k,2}(\Omega_{R}(x))} \leq CM\left( \frac{r}{r_0}\right) ^{\alpha_0}.
	\end{equation}
	Pick $r_1=r_0/(4C)^{1/\alpha}$ in the above, and we obtain
	$$
	\osc_{B_{r_1}(x_0)}u\leq \frac14M
	$$
	This together with \eqref{BMP02} yields
	\begin{equation}\label{BMP04}
		\inf\limits_{B_{r_1}(x_0)}|u-q|\geq  |u(x_0)-q|-\osc_{B_{r_1}(x_0)}u\geq \frac12M.
	\end{equation}
	By Lemma \ref{lemma:Courant-Lebesgue}, there exists $r_2\in(r_0,2r_0)$ such that 
	\begin{equation}\label{BMP05}
		\osc_{\partial B_{r_2}(x_0)\cap\Omega_R(x)}u\leq C\|u\|_{W^{k,2}(\Omega_{4R}(x))}.
	\end{equation}
	
	Note that $\partial B_{r_2}(x_0)\cap\partial\Omega_R(x)\not=\emptyset$. Using polar coordinates centered at $x_0$, we estimate
	\begin{equation*}
		\begin{split}
			\inf\Big\{|u(r_1,\theta)-&u(r_2,\theta)|:(r_i,\theta)\in\partial B_{r_i}(x_0)\cap\Omega_R(x), i=1,2\Big\}\\
			\leq& C\int_{S^{2k-1}}d\theta\int_{r_1}^{r_2} |u_r|\chi_{[r_1,r_2]\times S^{2k-1}}(r,\theta)dr\\
			\leq& \frac{C}{r_1^{2k-1}}\int_{S^{2k-1}}d\theta\int_{r_1}^{r_2} |u_r|\chi_{[r_1,r_2]\times S^{2k-1}}(r,\theta)r^{2k-1}dr\\
			\leq& \frac{C}{r_1^{2k-1}}\int_{B_{2r_0}(x)\cap\Omega_R(x)}|u_r|dx\\
			\leq& \frac{C}{r_1^{2k-1}}|B_{2r_0}(x)|^{\frac{2k-1}{2k}}\left( \int_{\Omega_R(x)}|\nabla u|^{2k}dx\right) ^{\frac{1}{2k}}\\
			\leq& C\frac{r_0^{2k-1}}{r_1^{2k-1}}\|u\|_{W^{k,2}(\Omega_{4R}(x))}
			\leq C\|u\|_{W^{k,2}(\Omega_{4R}(x))}.
		\end{split}
	\end{equation*}
This implies that there exists $\theta_0\in\partial B_1(x_0)$ such that 
	\begin{equation}\label{BMP07}
		|u(r_1,\theta_0)-u(r_2,\theta_0)|\leq C\|u\|_{W^{k,2}(\Omega_{4R}(x))}.
	\end{equation}
	Hence, by choosing an arbitrary $x^*\in \partial B_{r_2}(x_0)\cap\partial\Omega_R(x)$, we obtain from (\ref{BMP04}), (\ref{BMP05}) and (\ref{BMP07}) that
	\begin{align*}
		\frac{M}{2} \leq& \inf\limits_{B_{r_1}(x_0)}|u-q| \leq |u(r_1,\theta_0)-q|\\
		\leq& |u(r_1,\theta_0)-u(r_2,\theta_0)|+|u(r_2,\theta_0)-u(x^*)|+|u(x^*)-q|\\
		\leq& C\|u\|_{W^{k,2}(\Omega_{4R}(x))}+\osc_{\partial B_{r_2}(x_0)\cap\Omega_R(x)}u+\sup\limits_{\partial \Omega_R(x)}|u-q|\\
		\leq& C\left( \sup\limits_{\partial \Omega_R(x)}|u-q| + \|u\|_{W^{k,2}(\Omega_{4R}(x))}\right).
	\end{align*}
	The proof is complete.
\end{proof}

\subsection{Proof of Theorem \ref{thm:main}}
Now we are ready to prove Theorem \ref{thm:main}.

\begin{proof}[Proof of Theorem \ref{thm:main}]
	Let $x_0\in\partial\Omega$ and take $q= g(x_0)= u(x_0)$ in Proposition \ref{prop:BMP}. Note that
	\[
	\begin{aligned}
		\max\limits_{\partial\Omega_{R}(x_0)}|u-u(x_0)|&\leq \max\limits_{\partial\Omega_{R}(x_0)\cap\partial\Omega}|u-u(x_0)|+
	\osc_{\partial\Omega_{R}(x_0)\cap\Omega}u\\
	&=\max\limits_{\partial\Omega_{R}(x_0)\cap\partial\Omega}|g-g(x_0)|+
	\osc_{\partial\Omega_{R}(x_0)\cap\Omega}u.
	\end{aligned}
	\]
	The first term tends to 0 as $R\to 0$ since $g\in C(\partial\Omega)$. The second term tends to 0 as $R\to 0$ by Lemma \ref{lemma:Courant-Lebesgue}. This implies the continuity of $u$ as desired.

%
%
%
\end{proof}

\begin{remark}\label{rmk:on main result}
	The proof above extends to solutions to the following inhomogeneous elliptic system 
	\begin{equation}\label{eq:inhom D-G system}
		\Delta^k u = \sum_{l=0}^{k-1}\Delta^l\left\langle V_l,du\right\rangle+ \sum_{l=0}^{k-2}\Delta^l\delta(w_ldu)+f\quad \text{in } \Omega\subset \R^{2k}
	\end{equation} 
with $f\in L^p$ for some $p>1$ and \eqref{eq:Assumption 1}, \eqref{eq:Assumption 2}. Indeed, by \cite[Theorem 1.1]{GXZ2022}, in this case, the interior regularity estimate \eqref{eq:interior regularity} becomes
\begin{equation}\label{eq:inter regular with f}
	\osc_{B_{\tau r}(x)}u\lesssim \tau^{\alpha}\left(\|u\|_{W^{k,2}(B_{4r}(x_0))}+\|f\|_{L^p(B_{4r}(x_0))}\right)\quad \text{for }0<\tau\leq 1. 
\end{equation}
With this, the buondary maximal principle \eqref{eq:boundary mp} remains valid with an extra term $\|f\|_{L^p(\Omega_{4R}(x))}$ on the right hand side. The proof of Theorem \ref{thm:main} then works with obvious modifications.
\end{remark}

\subsection*{Acknowledgements}

The authors would like to Prof.~Chang-Lin Xiang and Chang-Yu Guo for posing this question to them and for many useful conservations.

\end{document}